\documentclass{paper}
\usepackage[utf8x]{inputenc}
\usepackage{amsmath, amsthm}
\usepackage{amsfonts,  mathbbol}
\usepackage{amssymb}
\usepackage{graphicx}
\usepackage{mathdots}
\usepackage[margin=3.5cm]{geometry}

 \newtheorem{theorem}{Theorem}[section]
\newtheorem{lemma}[theorem]{Lemma}
\newtheorem{corollary}[theorem]{Corollary}
 \theoremstyle{definition}
 \newtheorem{definition}[theorem]{Definition}
 
 \newtheorem{exampl}[theorem]{Example}

\newcommand{\Soc}{\operatorname{Soc}}

\newcommand{\spn}{\operatorname{span}}

\newcommand{\Tor}{\operatorname{Tor}}
\begin{document}

\title{On the Betti numbers and Rees algebras of ideals with linear powers}

\author{Lisa Nicklasson \\
\small{Max-Planck Institute for Mathematics in the Sciences}\\
\small{lisa.nicklasson@mis.mpg.de}}
%
\maketitle

\begin{abstract}
 An ideal $I \subset \mathbb{k}[x_1, \ldots, x_n]$ is said to have linear powers if $I^k$ has a linear minimal free resolution, for all integers $k>0$. In this paper we study the Betti numbers of $I^k$, for ideals $I$ with linear powers. We provide linear relations on the Betti numbers, which holds for all ideals with linear powers. This is especially useful for ideals of low dimension. The Betti numbers are computed explicitly, as polynomials in $k$, for the ideal generated by all square free monomials of degree $d$, for $d=2, 3$ or $n-1$, and the product of all ideals generated by $s$ variables, for $s=n-1$ or $n-2$. We also study the generators of the Rees ideal, for ideals with linear powers. Especially, we are interested in ideals for which the Rees ideal is generated by quadratic elements. This problem is related to a conjecture on matroids by White. 
\end{abstract}

\section{Introduction}

Let $S=\mathbb{k}[x_1, \ldots, x_n]$ be a polynomial ring over a field $\mathbb{k}$. For a homogeneous ideal $I \subset S$, recall that the numbers $\beta_1, \ldots, \beta_n$ in a minimal free resolution
\[
 0 \to S^{\beta_n} \to \dots \to S^{\beta_1} \to S \to S/I \to 0 
\]
are called the \emph{Betti numbers of $S/I$}. We will denote the $i$-th Betti number of $S/I$ by $\beta_i(S/I)$. In \cite{kodi} the following famous result is proved. 

\begin{theorem}
Let $I\subset S$ be a homogeneous ideal. The Betti numbers $\beta_i(S/I^k)$, seen as functions of $k$, are polynomials in $k$, for $k$ large enough. 
\end{theorem}

Given some nice class of ideals, one may ask if it is possible to give explicit expressions for these polynomials. In this paper we will attempt to answer this question for ideals with \emph{linear powers}. 

\begin{definition}
A homogeneous ideal $I \subseteq S$  is said to have \emph{linear resolution} if all the maps $\partial_i$ with $i>1$ in a minimal free resolution
\[
 0 {\to} S^{\beta_n} \overset{\partial_n}{\to} S^{\beta_{n-1}} \to  \dots \to S^{\beta_1} \overset{\partial_1}{\to} S \to S/I \to 0 
\]
of $S/I$ are represented by matrices whose non-zero entries have degree one. Moreover, the ideal $I$ has \emph{linear powers} if $I^k$ has linear resolution, for all $k$. 
\end{definition}

Two examples of ideals with linear powers, that will be of special interest in this paper, are the following two classes.

{\bf A:} \emph{The ideals generated by all square-free monomials of degree $d$.} 

{\bf B:} \emph{The products of all ideals generated by $s$-sized subsets of the variables. }

In Section \ref{sec:betti} we will see how the Betti numbers of ideals with linear powers can be computed using the Herzog-Kühl equations. Typically, this is a suitable approach for ideals of low dimension. Application of this result, together with a method involving Rees algebras, gives explicit expressions for $\beta_i(S/I^k)$, as polynomials in $k$, for the ideals in {\bf A} with $d=2,3,$ and $n-1$, and the ideals in {\bf B} with $s=n-1$ and $n-2$. 

Both {\bf A} and {\bf B} are subclasses of a larger class of ideals with linear powers, namely the \emph{polymatroidal ideals}.

\begin{definition}
 Let $I$ be a monomial ideal, and let $\mathcal G(I)$ be the minimal generating set of monomials. The ideal $I$ is called \emph{polymatroidal} if all the monomials in $\mathcal{G}(I)$ have the same degree and the following condition is satisfied. Let $x_1^{a_1} \cdots x_n^{a_n}$ and $x_1^{b_1} \cdots x_n^{b_n}$ be two monomials in $\mathcal G(I)$, such that $a_i > b_i$ for some $i$. Then there is a $j$ such that $a_j<b_j$ and $(x_j/x_i)x_1^{a_1} \cdots x_n^{a_n} \in \mathcal{G}(I)$. 
\end{definition}

It was proved in \cite{herz-tak} that polymatroidal ideals have linear resolution. Combined with the result from \cite{con-herz}, that products of polymatroidal ideals are polymatroidal, it follows that polymatroidal ideals have linear powers. Another class of ideals with linear powers, which is also found in \cite{con-herz}, is the products of ideals generated by linear forms. 

A third class of ideals with linear powers is a certain type of edge ideals of graphs. We recall the definition of an edge ideal. Let $G$ be a graph with the vertex set $V=\{x_1, \ldots, x_n\}$. The \emph{edge ideal} of $G$, denoted $I(G)$, is the monomial ideal generated by all $x_ix_j$ such that $\{x_i,x_j\}$ is an edge in $G$. The \emph{complementary graph} of $G$ is the graph with the same vertex set, and whose edges are all $\{x_i,x_j\}$ that are \emph{not} edges of $G$. A graph is called \emph{chordal} if every cycle of length at least four has a chord. It was proved in \cite{frob} that an edge ideal $I(G)$ has a linear resolution if and only if the complementary graph of $G$ is chordal. In fact, these ideals have linear powers. This follows from the result in \cite{herz-hibi-zhe}, which states that monomial ideals generated in degree two, with linear resolution, have linear powers.  

Recall that the \emph{Rees algebra} of an ideal $I \subset S$ is defined as
\[
 R(I) = \bigoplus_{j \ge 0} I^jt^j \subset S[t].
\]
If $I=(f_1, \ldots, f_m)$ we have $R(I) = S[f_1t, \ldots, f_mt]$. Let $T=S[y_1, \ldots, y_m]$, with the bigrading $\deg x_i= (1,0)$ and $\deg y_i=(0,1)$, and define a homomorphism $\phi:T \to S[t]$ by $\phi(x_i)=x_i$ and $\phi(y_i)=f_it$. Then the image of $\phi$ is $R(I)$, and hence $R(I) \cong T/J$ where $J= \ker \phi$. The ideal $J$ is called the \emph{Rees ideal} of $I$. If the $f_i$'s are all of the same degree, then $J$ is a homogeneous ideal. If the $f_i$'s are monomials, then $J$ is a binomial ideal. 

It was asked in \cite{bru-con} if the Rees ideal of a polymatroidal ideal is generated in the degrees $(0,2)$ and $(1,1)$. This question is a generalization of White's conjecture, \cite{white}, on matroids. For further discussion on White's conjecture, and its connection to polymatroidal ideals, see \cite{bru-con} and \cite{herz-hibi}. Inspired by this question, we devote Section \ref{sec:rees} to the study of ideals with linear powers, not necessarily polymatroidal, with the property that their Rees ideals are generated in the degrees $(0,2)$ and $(1,1)$.

\section{The Betti numbers of ideals with linear powers}\label{sec:betti}
\subsection{General result}
Let $S=\mathbb{k}[x_1, \ldots, x_n]$, and let $I \subset S$ be a homogeneous ideal with linear powers, generated in degree $d$. The Betti numbers of $S/I$ satisfy certain linear equations known as the \emph{Herzog-Kühl}-equations. These equations were first studied in \cite{herz-kuhl}. We will review how to obtain the Herzog-Kühl equations, for $I^k$. The ideal $I^k$ is generated in degree $dk$, and has the minimal graded free resolution \[
 0\!  \to S(-(dk+n-1))^{\beta_n^{(k)}} \!\!\!\to \dots  \to S(-(dk+1))^{\beta_2^{(k)}} \!\!\!\to S(-dk)^{\beta_1^{(k)}} \!\!\! \to S \to S/I^k \to 0 \ .
\]
The Hilbert series of $S/I^k$ is given by 
\[
 \frac{1+\sum_{i=1}^n (-1)^i \beta_i^{(k)} t^{dk+i-1}}{(1-t)^n} = \frac{p(t)}{(1-t)^{\delta}} \ ,
\]
where $p(t)$ is a polynomial, and $\delta=\dim I$. We rewrite this as
\[
 1+\sum_{i=1}^n (-1)^i \beta_i^{(k)}t^{dk+i-1} = (1-t)^{n-\delta}p(t) \ ,
\]
and we see that $t=1$ is a zero of the right hand side, and of its $n-\delta-1$ first derivatives. This gives us the $n-\delta$ Herzog-Kühl equations 
\begin{equation}\label{eq:hk_sys} 
  \left\{
 \begin{array}{rl}
  \sum_{i=1}^n(-1)^i\beta_i^{(k)} &=-1 \\[1ex]
  \sum_{i=1}^n(-1)^i(dk+i-1)\beta_i^{(k)} &=0 \\[1ex]
  \sum_{i=1}^n(-1)^i(dk+i-1)(dk+i-2)\beta_i^{(k)} &=0 \\[1ex]
  \vdots & \\
  \sum_{i=1}^n(-1)^i(dk+i-1)\cdots (dk+i-(n-\delta-1))\beta_i^{(k)} &=0 \ .\\[1ex]
 \end{array}
\right.
\end{equation}
This system of equations can actually be written on a very nice form, as we will see in Theorem \ref{thm:main} below. We will use the convention $\binom{j}{i}=0$ when $i>j$, for binomial coefficients. Also, we will use the notation $[c_{ij}]$ for the matrix with the element $c_{ij}$ on row $i$, column $j$. Similarly, we use the notation $[c_i]$ for column vectors.  

\begin{theorem}\label{thm:main}
Let $S=\mathbb{k}[x_1, \ldots, x_n]$. For a homogeneous ideal $I \subseteq S$ with linear powers, generated in degree $d$, define the column vector ${ \beta^{(k)}}=(\beta_i(S/I^k))_{1 \le i \le n}$. Then 
\[
 \left[(-1)^{i+j} \binom{j}{i}\right]_{\substack{ 0 \le i <n-\dim I \\ 0 \le j < n}} { \beta^{(k)}} =   \left[ \binom{dk+i-1}{i} \right]_{0 \le i < n- \dim I \ \textstyle{.}}
\]
\end{theorem}
\begin{proof}
We start by writing the system of equations (\ref{eq:hk_sys}) in matrix form, as
{\small \[
  \left[\begin{matrix}
  -1        & 1              & \cdots & (-1)^n \\[1ex]
  -dk       & dk+1           & \cdots & (-1)^n(dk+n-1) \\[1ex]
  -dk(dk-1) & \!\!\!(dk\!+\!1)dk & 	& (-1)^n(dk\!+\!n\!-\!1)(dk\!+\!n\!-\!2) \\
  \vdots    &	      &	& \vdots \\
  \!-dk(dk\!-\!1)\cdots (dk\!-\!n\!+\!\delta\!+\!2)& \dots & \dots & (-1)^n(dk\!+\!n\!-\!1) \cdots (dk\!+\!\delta-1)
 \end{matrix}\right. \!\!
 \left|\begin{matrix}
       -1 \\[1ex] 0\\[1ex] 0 \\ \vdots  \\ 0 
       \end{matrix}
\right]_{\ \textstyle{,}}
\]}
where $\delta=\dim I$. Let us index the rows by $i=0, 1, \ldots, n-\delta-1$. When we subtract a $(dk-i)$-multiple of row $i$ from row $i+1$, for $i=n-\delta-2, n-\delta-3, \ldots,  1, 0$ we get 
{\small \[
 \left[\begin{matrix}
  -1    & 1		& -1             & \cdots & (-1)^n \\[1ex]
  0 	& 1		& -2          	 & \cdots & (-1)^n (n-1) \\[1ex]
  0 	& dk+1  	& -2(dk+2)	 &	  & (-1)^n (n-1)(dk+n-1) \\[1ex]
  0	& (dk+1)dk	& \!\!\!\!\!\!\!\!\!\!\!\!\!\!-2(dk\!+\!2)(dk\!+\!1)\!\!\!\!\! &     	  & \!\!(-1)^n(n\!-\!1)(dk\!+\!n\!-\!1)(dk\!+\!n\!-\!2) \\
  \vdots&	        &		 & 	  & \vdots \\
  0 & \!\!\!\!\!\!\!(dk\!+\!1)dk\cdots  (dk\!-\!n\!+\!\delta\!+\!4)& \dots & \dots & \!\!\!\!\!(-1)^n(n\!-\!1)(dk\!+\!n\!-\!1) \cdots (dk\!+\!\delta)
 \end{matrix}\right. \!\!
 \left|\begin{matrix}
       -1 \\[1ex] dk \\[1ex] 0 \\[1ex] 0  \\ \vdots \\ 0 
       \end{matrix}
\right]_{\ \textstyle{,}}
\]}
Next, we subtract a $(dk-i+2)$-multiple of row $i$ from row $i+1$, for $i=n-d-1,  \ldots, 1$, which gives
{\small
\[
 \left[\begin{matrix}
  -1    & 1	& -1             & \cdots & (-1)^n \\[1ex]
  0 	& 1	& -2          	 & \cdots & (-1)^n (n-1) \\[1ex]
  0 	& 0  	& -2	 	 &	  & (-1)^n (n-1)(n-2) \\[1ex]
  0	& 0	& -2(dk+2) 	 &     	  & \!\!(-1)^n(n\!-\!1)(n-2)(dk\!+\!n\!-\!1) \\
  \vdots& \vdots	       	  &		 & 	  & \vdots \\
  0 &  0& -2(dk+2) \cdots (dk\!-\!n\!+\!\delta\!+\!6) & \dots & \!\!\!\!\!(-1)^n(n\!-\!1)(n\!-\!2)(dk\!+\!n\!-\!1) \cdots (dk\!+\!\delta\!-\!1)
 \end{matrix}\right. \!\!
 \left|\begin{matrix}
       -1 \\[1ex] dk \\[1ex] -dk(dk+1) \\[1ex] 0  \\ \vdots \\ 0 
       \end{matrix}
\right]_{\ \textstyle{.}}
\]}
Continuing in the same fashion eventually gives
{\footnotesize
\[
\left[ \begin{matrix}
\!-1 		& 1			& -1		& 1			& \dots	& (-1)^{n-\delta}				& \dots & (-1)^n \\[1ex]
0		& 1			& -2		& 3			& \dots	& (-1)^{n-\delta}(n-\delta-1)			& \dots & (-1)^n (n-1) \\[1ex]
0		& 0			& -2		& 3 \cdot 2		& 		& (-1)^{n\!-\delta}(n\!-\!\delta\!-\!1)(n\!-\!\delta\!-\!2)\!\!\!\!	&		& \!\!\!\!(-1)^n (n\!-\!1)(n\!-\!2) \\
0		& 0			& 0		& 3 \cdot 2	        &		&	\vdots 					&		& \vdots \\

0		& 0			& 0 		& 0			&		&							&		& \\[-1ex]
\vdots	        & \vdots	        & \vdots	& \vdots	 	&	\ddots \!\!\!\!\!\!\!\!\!\!\!\!\!\!\!\!	& (-1)^{n-\delta}(n-\delta-1)!		&		& \vdots \\
0		&	0		&	0		&	0	&		& (-1)^{n-\delta}(n-\delta-1)!		&		& (-1)^n(n\!-\!1) \cdots (\delta\!+\!1) 
\end{matrix} \right| 
\left. \begin{matrix}
-1 \\[1ex] dk \\[1ex] -dk(dk+1) \\[1ex] dk(dk+1)(dk+2) \\ \vdots  \\  \\ \!(-1)^{n\!-\!d+1} \!dk \cdots \!(dk\!+\!n\!-\!\delta)\!  
\end{matrix} \right]
\]}
Notice that we now have $(-1)^{j+1}j!/(j-i)!$ in position $(i,j)$ in the coefficient matrix. Multiplying each row by $(-1)^{i+1}/i!$ completes the proof.
  \end{proof}

We can see that Theorem \ref{thm:main} is most useful for ideals of low dimension, since the number of equations is $n-\dim I$. If we can compute $m$ of the Betti numbers $\beta_i(S/I^k)$, for some $m \ge \dim I$, we have enough information to solve the system of equations, and we get explicit expressions for $\beta_i(S/I^k)$, as polynomials in $k$. For any ideal $J \subseteq S$, we know that $\beta_1(S/J)$ is always the number of minimal generators of $J$. We can also compute $\beta_n(S/J)$ in the following way. Recall that the Betti numbers satisfy $\beta_i(S/J) = \dim_\mathbb{k} \Tor_i^S(\mathbb{k},S/J)$, see e.\,g.\ \cite[Proposition 1.3.1]{bruns-herz}. Using the properties of the Tor functor we have an isomorphism of $\mathbb{k}$-spaces
\[
\Tor_n^S(\mathbb{k},S/J) = \Tor_n^S(S/(x_1, \ldots, x_n),S/J) \cong \Soc(S/J),
\]
where 
\[ \Soc(S/J) = \{ f \in S/J \ | \ fx_i=0, \ \mbox{for all} \ i=1,  \ldots, n\} \ \]
is the \emph{socle} of $S/J$. It follows that $\beta_n(S/J) = \dim_\mathbb{k} \Soc(S/J)$.

\begin{theorem}\label{thm:main_dim1-2}
 Let $S=\mathbb{k}[x_1, \ldots, x_n]$, and let $I \subset S$ be a homogeneous ideal with linear powers, generated in degree $d$. Let $\beta_i^{(k)}=\beta_i(S/I^k)$. If $\dim I =1$, then
 \[
  \begin{bmatrix}
   \beta_2^{(k)} \\ \vdots \\ \vdots \\ \beta_n^{(k)} 
  \end{bmatrix}= 
\left[ \binom{i}{j} \right]_{\substack{ 0 \le i \le n-2 \\ \substack 0 \le j \le n-2}} \left[\beta_1^{(k)}- \binom{dk+i}{i} \right]_{0 \le i \le n-2 {\ \textstyle{.}}}
 \]
If $\dim I =2$, then 
 \[
  \begin{bmatrix}
   \beta_3^{(k)} \\ \vdots \\ \vdots \\ \beta_n^{(k)} 
  \end{bmatrix}= 
\left[ \binom{i}{j} \right]_{\substack{ 0 \le i \le n-3 \\ \substack 0 \le j \le n-3}} \left[\beta_1^{(k)} -(-1)^{n+i}\binom{n-2}{i}\beta_n^{(k)}   - \binom{dk+i}{i} \right]_{0 \le i \le n-3 {\ \textstyle{.}}}
 \]
\end{theorem}
\begin{proof}
 Subtracting $\beta_1^{(k)}$ from both sides of the first equation in Theorem \ref{thm:main} results in the system 
 \[
\begin{array}{l}
\begin{matrix}
\ \ \ \beta_2^{(k)} & \  \beta_3^{(k)} & \   \beta_4^{(k)} & \  \beta_5^{(k)} & \ \dots & \ \  \beta_n^{(k)}
\end{matrix} \\
\left[ \begin{matrix}
-1 & 1 & -1 & 1 & \dots & (-1)^{n-1} \\[1ex]
1 & - \binom{2}{1} & \binom{3}{1} & -\binom{4}{1} & \dots & (-1)^n \binom{n}{1}  \\[1ex] 
0 & 1 & -\binom{3}{2} & \binom{4}{2} & \dots & (-1)^{n+1}\binom{n}{2} \\[1ex]
0 & 0 & 1 & - \binom{4}{3} & & \vdots \\
\vdots & \vdots & \vdots & \vdots & & \\
0 & 0 & & & 
\end{matrix} \right.
\left| \begin{matrix}
1-\beta_1^{(k)} \\[1ex] dk \\[1ex] \binom{dk+1}{2} \\[1ex] \binom{dk+2}{3} \\ \vdots \\ \binom{dk+n-\delta}{n-\delta-1} \\
\end{matrix}\right]_{\ \textstyle{.}}
\end{array}
\]
Adding the first row to the second, then the second row  to the third, and so on, gives
 \[
\left[ \begin{matrix}
-1 & 1 & -1 & 1 & \dots & (-1)^{n-1} \\[1ex]
0 & - 1 & \binom{2}{1} & -\binom{3}{1} & \dots & (-1)^n \binom{n-1}{1}  \\[1ex] 
0 & 0 & -1 & \binom{3}{2} & \dots & (-1)^{n+1}\binom{n-1}{2} \\[1ex]
0 & 0 & 0 & - 1 & & \vdots \\
\vdots & \vdots & \vdots & \vdots & & \\
0 & 0 & & & 
\end{matrix} \right.
\left| \begin{matrix}
1-\beta_1^{(k)} \\[1ex] dk+1-\beta_1^{(k)} \\[1ex] \binom{dk+2}{2}-\beta_1^{(k)} \\[1ex] \binom{dk+3}{3}-\beta_1^{(k)} \\ \vdots \\ \binom{dk+n-\delta-1}{n-\delta-1}-\beta_1^{(k)} \\
\end{matrix}\right]_{\ \textstyle{,}}
\]
where the relation $\binom{j+1}{i}-\binom{j}{i-1}=\binom{j}{i}$ has been used. Multiplying each row by $-1$ results in the system 
\[
 \left[(-1)^{i+j} \binom{j}{i}\right]_{\substack{ 0 \le i <n-\dim I \\ 0 \le j \le n-2}} \left[ \beta_i^{(k)} \right]_{2 \le i \le n} =  \left[\beta_1^{(k)}- \binom{dk+i}{i} \right]_{0 \le i < n- \dim I {\ \textstyle{.}}}
\]
The inverse of the square matrix $\big[ (-1)^{i+j}\binom{j}{i} \big]$ with $0 \le i, j \le s$ is $\big[ \binom{j}{i} \big]$. For a proof of this fact, see e. g. \cite{pascal}. This proves the case $\dim I =1$. Moving all multiples of $\beta_n^{(k)}$ to the right hand side of the above equation gives
\begin{align*}
 \left[(-1)^{i+j} \binom{j}{i}\right]_{\substack{ 0 \le i <n-\dim I \\ 0 \le j \le n-3}} & \left[ \beta_i^{(k)}  \right]_{2 \le i \le n-1} \! \\
 &=  \left[\beta_1^{(k)} \!-(-1)^{n+i}\binom{n-2}{i}\beta_n^{(k)}  - \binom{dk+i}{i} \right]_{0 \le i < n- \dim I {\ \textstyle{.}}}
\end{align*}
If $\dim I =2$, we now have a square coefficient matrix, and multiplying by its inverse completes the proof. 
  \end{proof}

\subsection{Application to square-free monomial ideals}\label{sec:sqfree}
Let $I$ be the ideal in $S=\mathbb{k}[x_1, \ldots, x_n]$, generated by all square-free monomials of degree $d$, for some fixed integer $d$, and $\beta_i^{(k)}=\beta_i(S/I^k)$. Such ideals are polymatroidal, and therefore have linear powers. We have $\dim I =d-1$. The $k$-th power $I^k$ is generated by all monomials $x_1^{a_1} \cdots x_n^{a_n}$ of degree $dk$, such that $a_i\le k$, and the number of such monomials can be computed using the principle of inclusion-exclusion. A monomial $x_1^{a_1} \cdots x_n^{a_n}$ of degree $dk$ with $a_{i_1}, \ldots, a_{i_m}>k$, for some indices $i_1, \ldots, i_m$, can be considered as a monomial of degree $dk-m(k+1)$, multiplied by $(x_{i_1}\cdots x_{i_m})^{k+1}$. It follows that the number of monomials $x_1^{a_1} \cdots x_n^{a_n}$ with $a_i>k$, for at least $m$ of the $a_i$'s is
\[ \binom{n}{m} \binom{n+dk-m(k+1)-1}{n-1}_{\ \textstyle{,}}\]
and hence
\begin{equation}\label{eq:beta1_square-free}
\beta_1^{(k)}=\sum_{\substack{m \ge 0 \\ m(k+1) \le dk}}\!\!(-1)^m \binom{n}{m} \binom{n+dk-m(k+1)-1}{n-1}_{\ \textstyle{.}}
\end{equation}
To compute $\Soc(S/I^k)$ it is sufficient to determine all monomials $f$ such that $f \notin I^k$ but $fx_i \in I^k$, for all $i$. Such a monomial $f$ must be of degree at least $dk-1$. Suppose $f=x_1^{b_1} \cdots x_n^{b_n}$ is a monomial of degree $dk$ or higher, with this property. We can factorize $f$ as 
\[
 f=gx_1^{a_1} \cdots x_n^{a_n} \ \mbox{where} \ g= \prod_{b_i \ge k}x_i^{b_i-k} \ \mbox{and} \ a_i=\begin{cases}
                                                                                                      b_i & \mbox{for} \ b_i\le k, \\
                                                                                                      k & \mbox{otherwise.}
                                                                                                     \end{cases}
\]
Since $f \notin I^k$ we have $\sum a_i <dk$. Then, for any $i$ such that $x_i | g$, we have $fx_i \notin I^k$. Hence, all elements of $\Soc(S/I^k)$ are of degree $dk-1$. More precisely 
\[
 \Soc(S/I^k) = \spn \{ x_1^{a_1} \cdots x_n^{a_n} \ | \ a_i<k, \sum_{i=1}^n a_i = dk-1 \},
\]
and we get
\begin{align}\label{eq:betan_square-free}
 \beta_n^{(k)} = \dim_{\mathbb{k}}\Soc(S/I^k)  &= \sum_{\substack{m \ge 0 \\ mk \le dk-1}} \!\!(-1)^m\binom{n}{m} \binom{n+dk-1-mk-1}{n-1} \\
 &= \sum_{m=0}^{d-1}(-1)^m \binom{n}{m} \binom{n+(d-m)k-2}{n-1}_{\ \textstyle{.}}
\end{align}

If $d=2,$ or $3$, we can now use Theorem \ref{thm:main_dim1-2} to compute the remaining Betti numbers.

\begin{corollary}\label{cor:sqfree_d2}
 Let $S=\mathbb{k}[x_1, \ldots, x_n]$, and let $I \subseteq S$ be the ideal generated by all square-free monomials of degree two. Let $\beta_i^{(k)}=\beta_i(S/I^k)$. Then
 \[ \beta_1^{(k)}= \binom{n+2k-1}{n-1} -n\binom{n+k-2}{n-1}_{\ \textstyle{,}} \]
 and
 \[ \begin{bmatrix}
\beta_2^{(k)} \\ \vdots \\ \vdots \\ \beta_n^{(k)}
\end{bmatrix} = \left[ \binom{j}{i} \right]_{\substack{0 \le i \le n-2 \\ 0 \le j \le n-2}} \left[\beta_1^{(k)}- \binom{2k+i}{i} \right]_{0 \le i \le n-2 {\ \textstyle{.}}}\]
\end{corollary}
\begin{proof}
 The expression for $\beta_1^{(k)}$ follows from substituting $d=2$ in (\ref{eq:beta1_square-free}). Since $\dim I =1$, Theorem \ref{thm:main_dim1-2} can be used to compute the remaining Betti numbers. 
  \end{proof}

\begin{exampl}\label{ex:sqfree_betti_deg2}
 Let $S=\mathbb{k}[x_1, x_2, x_3, x_4]$, and let $I \subseteq S$ be the ideal generated by all square-free monomials of degree two. By Corollary \ref{cor:sqfree_d2}, the Betti numbers of $S/I^k$ can now be expressed as the following polynomials in $k$ 
 \[
  \beta_1^{(k)} = \binom{3+2k}{3} - 4 \binom{k+2}{3} = \frac{1}{3}(2k^3+6k^2+7k+3),
 \]
and
\begin{align*}
 \begin{bmatrix}
  \beta_2^{(k)} \\ 
  \beta_3^{(k)} \\
  \beta_4^{(k)} \\
 \end{bmatrix}
&= \begin{bmatrix}
   1 & 1 & 1 \\
   0 & 1 & 2 \\
   0 & 0 & 1
  \end{bmatrix}
\begin{bmatrix}
 \beta_1^{(k)} - 1 \\
 \beta_1^{(k)} - (2k+1) \\
 \beta_1^{(k)} - (2k+1)(k+1) \\
\end{bmatrix} \\
&= \frac{1}{3} \begin{bmatrix}
   1 & 1 & 1 \\
   0 & 1 & 2 \\
   0 & 0 & 1
  \end{bmatrix}
\begin{bmatrix}
 2k^3+6k^2+7k \\
 2k^3+6k^2+k \\
 2k^3-2k \\
\end{bmatrix} = 
 \begin{bmatrix}
  2k^3+4k^2+k \\
  2k^3+2k^2-k \\
  \frac{2}{3}(k^3-k)
 \end{bmatrix}_{\ \textstyle{.}}
 \end{align*}
\end{exampl}

\begin{corollary}\label{cor:sqfree_d3}
 Let $S=\mathbb{k}[x_1, \ldots, x_n]$, and let $I \subseteq S$ be the ideal generated by all square-free monomials of degree three. Let $\beta_i^{(k)}=\beta_i(S/I^k)$. Then
 \[ \beta_1^{(k)} = \binom{n+3k-1}{n-1} -n\binom{n+2k-2}{n-1}+ \binom{n}{2}\binom{n+k-3}{n-1}_{\ \textstyle{,}} \]
\[ \beta_n^{(k)} = \binom{n+3k-2}{n-1} -n\binom{n+2k-2}{n-1}+ \binom{n}{2}\binom{n+k-2}{n-1} \]
and 
\[
 \begin{bmatrix}
\beta_2^{(k)} \\ \vdots \\ \vdots \\ \beta_{n-1}^{(k)}
\end{bmatrix} = \left[ \binom{j}{i} \right]_{\substack{0 \le i \le n-3 \\ 0 \le j \le n-3}} \left[ \beta_1^{(k)} - (-1)^{n+i}\beta_n^{(k)} \binom{n-2}{i}-\binom{3k+i}{i} \right]_{0 \le i \le n-3 {\ \textstyle{.}}}
\]
 \end{corollary}
\begin{proof}
 The expressions for $\beta_1^{(k)}$ and $\beta_n^{(k)}$ is obtained by substituting $d=3$ in (\ref{eq:beta1_square-free}) and (\ref{eq:betan_square-free}). Since $\dim I =2$, Theorem \ref{thm:main_dim1-2} can be used to compute the remaining Betti numbers. 
  \end{proof}

\begin{exampl}
 Let $S=\mathbb{k}[x_1, \ldots, x_5]$, and let $I \subseteq S$ be the ideal generated by all square-free monomials of degree three. By Corollary \ref{cor:sqfree_d3}, the Betti numbers of $S/I^k$ can now be expressed as the following polynomials in $k$ 
 \begin{align*}
  \beta_1^{(k)} &= \binom{4+3k}{4} -5 \binom{3+2k}{4}+ \binom{5}{2}\binom{2+k}{4} \\
  &= \frac{1}{24}(11k^4+50k^3+85k^2+70k+24),
    \\
  \beta_5^{(k)} &= \binom{3+3k}{4} -5 \binom{3+2k}{4}+ \binom{5}{2}\binom{3+k}{4}\\
  &= \frac{1}{24}(11k^4-18k^3-11k^2+18k), 
   \end{align*}
and
\begin{align*}
 \begin{bmatrix}
  \beta_2^{(k)} \\ 
  \beta_3^{(k)} \\
  \beta_4^{(k)} \\
 \end{bmatrix}
&= \begin{bmatrix}
   1 & 1 & 1 \\
   0 & 1 & 2 \\
   0 & 0 & 1
  \end{bmatrix}
\begin{bmatrix}
 \beta_1^{(k)} + \beta_5^{(k)} - 1 \\
 \beta_1^{(k)} -3\beta_5^{(k)}  - (3k+1) \\
 \beta_1^{(k)} + 3\beta_5^{(k)}  - \frac{(3k+2)(3k+1)}{2} \\
\end{bmatrix} 
=
\begin{bmatrix}
\frac{11}{6}k^4+\frac{11}{2}k^3+\frac{17}{3}k^2+2k  \\[0.7ex]
 \frac{11}{4}k^4+4k^3+\frac{1}{4}k^2-k \\[0.7ex]
 \frac{11}{6}k^4-\frac{1}{6}k^3-\frac{7}{3}k^2+\frac{2}{3}k
\end{bmatrix}_{\ \textstyle{.}}
\end{align*}

\end{exampl}

\subsection{Application to the product of ideals generated by variables}\label{sec:subsets}
In this section let $I$ be the product of all ideals generated by $s$-sized subsets of the variables, i. e. 
\[
 I = \prod_{1 \le i_1 < \dots < i_s \le n} (x_{i_1}, \ldots, x_{i_s}),
\]
for some $s \le n$. Ideals generated by variables are polymatroidal. Since products of polymatroidal ideals are polymatroidal, $I$ is polymatroidal and hence has linear powers. The generators are monomials of degree $\binom{n}{s}$, and $\dim I = n-s$. As in the previous section, we will consider the two cases $\dim I=2$ and $3$, which here means $s=n-1$ and $n-2$. 

Let us start with the case $s=n-1$. Then $I$ is the product of all ideals $(x_{i_1}, \ldots, x_{i_{n-1}})$, and it can easily be seen that $I$ is generated by all monomials of degree $n$, except the pure powers $x_1^n, x_2^n, \ldots, x_n^n$. It follows that $I^k$ is generated by all monomials $x_1^{a_1} \cdots x_n^{a_n}$ of degree $nk$ with $a_i \le k(n-1)$ for all $i$. With this information we can now compute $\beta_1(S/I^k)$ and then apply Theorem \ref{thm:main_dim1-2} to get the other Betti numbers.

\begin{corollary}
Let $S =\mathbb{k}[x_1, \ldots, x_n]$,
\[
 I= \prod_{1 \le i_1 < \dots < i_{n-1} \le n} (x_{i_1}, \ldots, x_{i_{n-1}}) \subset S ,
\]
and let $\beta_i^{(k)}=\beta_i(S/I^k)$. Then 
\[
 \beta_1^{(k)} = \binom{n(k+1)-1}{n-1}-n \binom{n+k-2}{n-1}_{\ \textstyle{,}}
\]
and
\[ \begin{bmatrix}
\beta_2^{(k)} \\ \vdots \\ \vdots \\ \beta_n^{(k)}
\end{bmatrix} = \left[ \binom{j}{i} \right]_{\substack{0 \le i \le n-2 \\ 0 \le j \le n-2}} \left[\beta_1^{(k)}- \binom{2k+i}{i} \right]_{0 \le i \le n-2 {\ \textstyle{.}}}
\]
\end{corollary}
\begin{proof}
The number of monomials of degree $nk$ is $\binom{n+nk-1}{n-1}$. The monomials of degree $nk$ that have $a_i>k(n-1)$, for some $i$, can be written as a monomial of degree $k-1$ multiplied by $x_i^{k(n-1)+1}$, and the number of such monomials is $n \binom{n+k-2}{n-1}$. Since only one variable can have degree greater than $k(n-1)$, without making the total degree too high, we can now conclude that
\[
 \beta_1^{(k)} = \binom{n+nk-1}{n-1}-n \binom{n+k-2}{n-1}_{\ \textstyle{.}}
\]
Since $\dim I = 1$, the second part follows from Theorem \ref{thm:main_dim1-2}.
  \end{proof}

Let us now assume $ n\ge 3$, and consider the case $s=n-2$. 
\begin{lemma}
 The ideal
 \[
  \prod_{1 \le i_1 < \dots < i_{n-2} \le n} (x_{i_1}, \ldots, x_{i_{n-2}})
\]
is generated by all monomials $x_1^{a_1} \cdots x_n^{a_n}$ of degree $\binom{n}{2}$ such that $a_i \le \binom{n-1}{2}$, for $i=1, \ldots, n$, where at least three $a_i$'s are positive. 
\end{lemma}
\begin{proof}
Let us denote the ideal by $I$. We can see that $a_i \le \binom{n-1}{2}$ for any $x_1^{a_1} \cdots x_n^{a_n} \in I$, since there are only $\binom{n-1}{2}$ different ideals $(x_{i_1}, \ldots, x_{i_{n-2}})$ containing a specific variable $x_i$. Also, a monomial $x_1^{a_1}x_2^{a_2}$ can not be contained in $I$, since the ideal $(x_3, \ldots, x_n)$ is one of the ideals in the product defining $I$. The same argument holds for any other monomial in two variables. Hence any $x_1^{a_1} \cdots x_n^{a_n} \in I$ must have at least three positive $a_i$'s.  We must now prove that all monomials satisfying these conditions are contained in $I$. 

Let $f=x_1^{a_1} \cdots x_n^{a_n}$ be a monomial of degree $\binom{n}{2}$ such that $a_i \le \binom{n-1}{2}$, for all $i$, and at least three $a_i$'s are positive. We want to prove that $f \in I_{n-2}$. This will be proved by induction over $\max_i(a_i)$, starting from the highest possible value $\max(a_i)=\binom{n-1}{2}$. Without loss of generality, we may, in the base case, assume that $a_1= \binom{n-1}{2}$. Notice that $I=J_1 J_2$ for
\[
 J_1=\prod_{2 \le i_1 < \dots < i_{n-2} \le n} \!\!(x_{i_1}, \ldots, x_{i_{n-2}}) \ \ \mbox{and} \ \ J_2= \prod_{2 \le i_1 < \dots < i_{n-3} \le n} \!\! (x_1,x_{i_1}, \ldots, x_{i_{n-3}}) \ .
\]
The ideal $J_1$ is generated by all monomials of degree $n-1$ in the variables $x_2, \ldots, x_n$, except the pure powers. In particular $x_2^{a_2} \cdots x_{n}^{a_{n}}$ is one of the generators. It is clear that $x_1^{\binom{n-1}{2}}=x_1^{a_1}$ is one of the generators in $J_2$, and it follows that $f=x_1^{a_1} \cdots x_n^{a_n}$ is one of the generators in $I$. 

Next, suppose that $f$ is such that $\max_i(a_i)=m< \binom{n-1}{2}$, and assume that the statement is true for all monomials $x_1^{b_1} \cdots x_n^{b_n}$ with $\max_i(b_i)>m$. Without loss of generality, we may assume that $a_1=m$.  First, notice that there is some $i$ such that $\frac{x_1}{x_i}f$ is a monomial satisfying our conditions, i. e. containing at least three variables. Indeed, if such an $i$ would not exist, then $f$ would be on the form $f=x_1^mx_ix_j$. But this would imply
\[ m=\binom{n}{2}-2 \ge \binom{n}{2} - \binom{n-1}{1} = \binom{n-1}{2}_{\ \textstyle{,}}\]
contradicting $m< \binom{n-1}{2}$. Again without loss of generality, we assume that 
\[g=\frac{x_1}{x_2}f=x_1^{m+1}x_2^{a_2-1}x_3^{a_3} \cdots x_n^{a_n}\]
is a monomial in at least three variables. By the induction hypothesis $g \in I$. Then we also have $x_1^{a_2-1}x_2^{m+1}x_3^{a_3} \cdots x_n^{a_n} \in I,$
by symmetry. Since $I$ is polymatroidal and $m+1>a_2-1$, we can conclude that $f=\frac{x_2}{x_1}g \in I$. 
  \end{proof}

It follows that $I^k$ is generated by all monomials $x_1^{a_1} \cdots x_n^{a_n}$ of degree $k \binom{n}{2}$ such that $a_i \le k\binom{n-1}{2}$ and at least three $a_i$'s are positive. In a similar way as for the square-free monomials ideals, we can see that the socle elements must have degree $k\binom{n}{2}-1$. It follows that the socle is spanned by the monomials $x_1^{a_1} \cdots x_n^{a_n}$ of degree $k \binom{n}{2}-1$ such that $a_i < k\binom{n-1}{2}$ and at least three $a_i$'s positive. We can now use Theorem \ref{thm:main_dim1-2} to compute the Betti numbers of the product of all ideals generated by $n-2$ variables.

\begin{corollary}
  Let $S=\mathbb{k}[x_1, \ldots, x_n]$, and
 \[
 I= \prod_{1 \le i_1 < \dots < i_{n-2} \le n} (x_{i_1}, \ldots, x_{i_{n-2}}) \subset S.
\]
Let $\beta_i^{(k)}=\beta_i(S/I^k)$. 
  Then
 \[ \beta_1^{(k)} = \binom{n+k\binom{n}{2}-1}{n-1} -n\binom{n+k(n-1)-2}{n-1}- \binom{n}{2}\Big(1+k(n-1)\Big( \frac{n}{2}-1 \Big)\Big), \]
\[ \beta_n^{(k)} = \binom{n+k\binom{n}{2}-2}{n-1} -n\binom{n+k(n-1)-2}{n-1}- \binom{n}{2}\Big(1+k(n-1)\Big( \frac{n}{2}-1 \Big)\Big) \]
and 
\[
 \begin{bmatrix}
\beta_2^{(k)} \\ \vdots \\ \vdots \\ \beta_{n-1}^{(k)}
\end{bmatrix} = \left[ \binom{j}{i} \right]_{\substack{0 \le i \le n-3 \\ 0 \le j \le n-3}} \left[ \beta_1^{(k)} - (-1)^{n+i}\beta_n^{(k)} \binom{n-2}{i}-\binom{k\binom{n}{2}+i}{i} \right]_{0 \le i \le n-3}
\]
\end{corollary}

\begin{proof}
 To find $\beta_1^{(k)}$ we shall compute the number of monomials $x_1^{a_1} \cdots x_n^{a_n}$ of degree $k \binom{n}{2}$ such that $a_i \le k\binom{n-1}{2}$, with at least three positive $a_i$'s. A monomial $x_1^{a_1} \cdots x_n^{a_n}$ of degree $k \binom{n}{2}$ can have at most one $a_i > k\binom{n-1}{2}$, or the total degree will be too high. It follows that the number of monomials $x_1^{a_1} \cdots x_n^{a_n}$ of degree $k \binom{n}{2}$ such that $a_i \le k\binom{n-1}{2}$ is
 \[
  \binom{n+k\binom{n}{2}-1}{n-1} -n\binom{n+k(n-1)-2}{n-1}_{\ \textstyle{.}}
 \]
 From this we must subtract the number of monomials 
 \[
  x_i^{a_i}x_j^{k \binom{n}{2}-a_i} \ \mbox{such that} \ a_i \le k \binom{n-1}{2}, \ k \binom{n}{2}-a_i\le k \binom{n-1}{2}_{\ \textstyle{.}}
 \]
The condition on $a_i$ simplifies to 
\[
 k(n-1) \le a_i \le  k \binom{n-1}{2}_{\ \textstyle{,}} 
\]
so for each pair of variables $x_i, x_j$ there are 
\[
 k \binom{n-1}{2}-k(n-1)+1=1+k(n-1)\Big( \frac{n}{2}-1 \Big)
\]
such monomials. This proves 
 \[ \beta_1^{(k)} = \binom{n+k\binom{n}{2}-1}{n-1} -n\binom{n+k(n-1)-2}{n-1}- \binom{n}{2}\Big(1+k(n-1)\Big( \frac{n}{2}-1 \Big)\Big)_{\ \textstyle{.}} \]
 To find $\beta_n^k$ we shall compute the number of monomials $x_1^{a_1} \cdots x_n^{a_n}$ of degree $k \binom{n}{2}-1$ such that $a_i < k\binom{n-1}{2}$ and at least three $a_i$'s positive. The same argument as above results in 
 \[ \beta_n^{(k)} = \binom{n+k\binom{n}{2}-2}{n-1} -n\binom{n+k(n-1)-2}{n-1}- \binom{n}{2}\Big(1+k(n-1)\Big( \frac{n}{2}-1 \Big)\Big). \]
 The expression for the remaining Betti numbers follows directly from Theorem \ref{thm:main_dim1-2}. 
  \end{proof}

\section{The Rees algebra of monomial ideals with linear powers}\label{sec:rees}
The object of this section is to study the generators of the Rees ideal, for certain monomial ideals with linear powers. 

For a bigraded module $M$, we let $M_{(i,j)}$ denote the graded component of elements of degree $(i,j)$. We also let $M_{(*,j)}=\bigoplus_{i\ge 0} M_{(i,j)}$, and we say that the elements of $M_{(*,j)}$ are of degree $(*,j)$, and analogously for $M_{(i,*)}$.   

\subsection{The degrees of the generators of the Rees ideal}
We start by the following observation on the Rees ideal of ideals with linear resolution. 

\begin{lemma}\label{lemma:rees1}
 Let $I$ be an ideal such that $I^k$ has a linear resolution, and let $J$ be the Rees ideal of $I$. Then the elements of degree $(s,k)$, with $s \ge 1$, in $J$ are generated by the elements of degree $(1,k)$ in $J$. 
\end{lemma}
\begin{proof}
 Suppose $I=(f_1, \ldots, f_m)$, and let $h\in J$ be homogeneous of degree $(s,k)$. Then $h=\sum_{i=0}^Ng_ih_i$, where $g_i\in S$, and the $h_i$'s are the monomials of degree $k$ in the variables $y_1, \ldots, y_m$. The monomials $h_i$ are mapped to a (not necessarily minimal) generating set of $I^k$ under $\phi$. By definition of the Rees ideal, $\phi(h)=0$, which implies $\sum_{i=1}^mg_i\phi(h_i)=0$. But since $I^k$ has a linear resolution, all relations on the generators are generated by relations in degree one. It follows that $h$ is generated by elements of degree $(1,k)$. 
  \end{proof}

An ideal $I$ is said to be of \emph{fiber type} if its Rees ideal is generated in degrees $(0,*)$ and $(*,1)$. This terminology was introduced in \cite{herz-hibi-vla}, where they prove that polymatroidal ideals are of fiber type. By \cite[Theorem 3.1]{villa} edge ideals of graphs are also of fiber type. 

\begin{lemma}
Let $I$ be an ideal of fiber type with linear resolution. Then the Rees ideal of $I$ is generated in degrees $(0,*)$ and $(1,1)$.
\end{lemma}

\begin{proof}
By Lemma \ref{lemma:rees1} the elements if degree $(*,1)$ in the Rees ideal are generated by elements of degree $(1,1)$. Since $I$ is of fiber type it follows that the Rees ideal is generated in degrees $(0,*)$ and $(1,1)$.  
  \end{proof}

We are especially interested in monomial ideals with linear resolution where the elements of degree $(0,*)$ in the Rees ideal are generated in degree $(0,2)$. Notice that the $(0,*)$-part of the Rees ideal is the same as the toric ideal defined by the monomial generators. More precisely, if $I$ is generated by monomials $f_1, \ldots, f_m$ define a map $\xi: \mathbb{k}[y_1, \ldots, y_m] \to S$ by $\xi(y_i)=f_i$. Then the $(0,*)$-part of the Rees ideal of $I$ is given by $\ker \xi$. It was proved in \cite{herz-hibi} that polymatroidal ideals with the so called \emph{strong exchange property} have toric ideal generated in degree 2. 

\begin{definition}
  A polymatroidal ideal $I$ is said to have the \emph{strong exchange property (SEP)} if it satisfies the following condition. Let $\mathcal{G}(I)$ be the generating set of monomials. If $x_1^{a_1} \cdots x_n^{a_n}$ and $x_1^{b_1} \cdots x_n^{b_n}$ are elements of $\mathcal{G}(I)$, such that $a_i>b_i$ and $a_j<b_j$, then $(x_j/x_i)x_1^{a_1} \cdots x_n^{a_n} \in \mathcal{G}(I)$.  
\end{definition}

An example of polymatroidal ideals with the SEP are the ideals generated by all square-free monomials in some degree, that we saw in Section \ref{sec:sqfree}. The generators of the Rees ideal reflects the ''exchange relations''.

\begin{exampl}
 We return to algebra $S=\mathbb{k}[x_1,x_2,x_3,x_4]$ and the square-free monomial ideal
 \[
  I=(x_1x_2,x_1x_3,x_1x_4,x_2x_3,x_2x_4,x_3x_4) \subset S
 \]
from Example \ref{ex:sqfree_betti_deg2}. Let $T=S[y_{12},y_{13},y_{14},y_{23},y_{24},y_{34}]$, and define $\phi:T \to S[t]$ by $x_i \mapsto x_i$ and $y_{ij} \mapsto x_ix_j t$. Then the Rees ideal $J= \ker \phi$ is generated by the binomials
\begin{align*}
&x_3y_{12}-x_2y_{13}, \ x_1y_{23}-x_3y_{12}, \ x_3y_{14}-x_4y_{13}, \ x_2y_{14}-x_4y_{12}, \\
&x_3y_{24}-x_4y_{23}, \ x_1y_{24}-x_4y_{12}, \ x_2y_{34}-x_4y_{23}, \ x_1y_{34}-x_4y_{13} 
\end{align*}
of degree $(1,1)$ and 
\[ y_{12}y_{34}-y_{13}y_{24}, \ y_{12}y_{34}-y_{14}y_{23} \]
of degree $(0,2)$. 
\end{exampl}

Another class of polymatroidal ideals with toric ideal generated in degree 2 are the \emph{transversal} ideals, see \cite[Theorem 3.5]{conca}. A polymatroidal ideal is called transversal if it is the product of ideals generated by variables. So in particular, the ideals studied in Section \ref{sec:subsets} are transversal.
 Next, we will see a class of edge ideals with toric ideal generated in degree 2. 

Let $G=(V,E)$ be a bipartite graph, with the partition $V=\{x_1, \ldots, x_n\} \cup \{y_1, \ldots, y_m\}$ of the vertices. Let $e_{ij}$ denote the edge $\{x_i,y_j\}$. The graph $G$ is called a \emph{Ferrers graph} if, after possibly re-indexing of the vertices, 
\[
 e_{ij} \in E \ \ \implies \ \ e_{rs} \in E \ \mbox{for all} \ 1 \le r \le i, \ \mbox{and} \ 1 \le s \le j.
\]
The name comes from the fact that the graph can be represented by a Ferrers diagram. The edge ideal of a Ferrers graph is called a \emph{Ferrers ideal}. It is proved in \cite{cor-nag} that the edge ideal of a bipartite graph has a linear resolution if and only if it is a Ferrers ideal.

\begin{lemma}\label{lemma:ferrer1}
 Let $J$ be the Rees ideal of a Ferrers ideal. The elements of degree $(0,*)$ in $J$ are generated by the elements of degree $(0,2)$ in $J$. 
\end{lemma}

\begin{proof}
 The elements of degree $(0,s)$ in $J \subseteq T$ are linear combinations of elements like
 \[
  f=e_{i_1 j_1} \cdots e_{i_s j_s} -e_{i_1 k_1} \cdots e_{i_s k_s} \ \mbox{where} \ \{j_1, \ldots, j_s\} = \{k_1, \ldots k_s\}. 
 \]
Let $L$ be the ideal generated by all elements on the form  $e_{i_1j_1}e_{i_2j_2}-e_{i_1j_2}e_{i_2j_1}$. We want to prove that $f=0$ in $T/L$. We may assume that $i_1 \ge i_2 \dots \ge i_s$, and $j_1 \ge k_1$. Since $\{j_1, \ldots, j_s\} = \{k_1, \ldots k_s\}$ there is an index $r$ such that $k_r=j_1$. It follows that $e_{i_rk_1}\in E$, since $e_{i_1j_1} \in E$, and $i_r \le i_1$ and $k_1 \le j_1$. Then $ e_{i_1k_1}e_{i_rk_r} =e_{i_ij_1}e_{i_rk_1}$ in $T/L$, and
\[
 f=e_{i_1 j_1} \cdots e_{i_s j_s} -e_{i_1 k_1} \cdots e_{i_s k_s} = e_{i_1j_1}(e_{i_2 j_2} \cdots e_{i_s j_s} -e_{i_2 \hat k_2} \cdots e_{i_s \hat k_s}),
\]
where $\hat k_i = k_i$ for $i \ne r$, and $\hat k_r = k_1$. The result now follows by induction over $s$. 
  \end{proof}

To summarize, we now have the following theorem. 

\begin{theorem}\label{thm:rees}
If $I$ is a polymatroidal ideal with the SEP, a transversal polymatroidal ideal, or a Ferrers ideal, then the Rees ideal of $I$ is generated in degrees $(0,2)$ and $(1,1)$. 
\end{theorem}

As pointed out in the introduction, there is reason the believe that Theorem \ref{thm:rees} holds for all polymatroidal ideals. Theorem \ref{thm:rees} does \emph{not} hold for all edge ideals with linear resolution, as the below example shows. For more details on the Rees algebra of an edge ideal see \cite{villa}.

\begin{exampl}
 Let $G$ be the graph in Figure \ref{fig:graph}. As we can see in Figure \ref{fig:graph_comp}, the complementary graph is chordal, so the edge ideal 
 \[
 I(G)=(x_1x_2,x_1x_3,x_2x_3,x_2x_4,x_2x_5,x_3x_5,x_3x_6,x_4x_5,x_5x_6)
  \]
 has linear powers. However, the Rees ideal is generated in the degrees $(1,1),(0,2)$ and $(0,3)$. The generator of degree $(0,3)$ corresponds to the relation 
 \[
  x_1x_2 \cdot x_3x_6 \cdot x_4x_5 = x_1x_3 \cdot x_2x_4 \cdot x_5x_6 \ . 
 \]
Notice also that $I(G)$ is not polymatroidal. To realize this, consider the two monomials $x_1x_2$ and $x_5x_6$, that are in the generating set of $I(G)$. The monomial $x_1x_2$ has higher $x_2$-degree than $x_5x_6$, but neither $(x_5/x_2)x_1x_2=x_1x_5$, nor $ (x_6/x_2)x_1x_2=x_1x_6$ belongs to $I(G)$. 
 
 \begin{figure}[ht]
 \begin{minipage}{0.35\textwidth}
  \centering
  \includegraphics[scale=0.3]{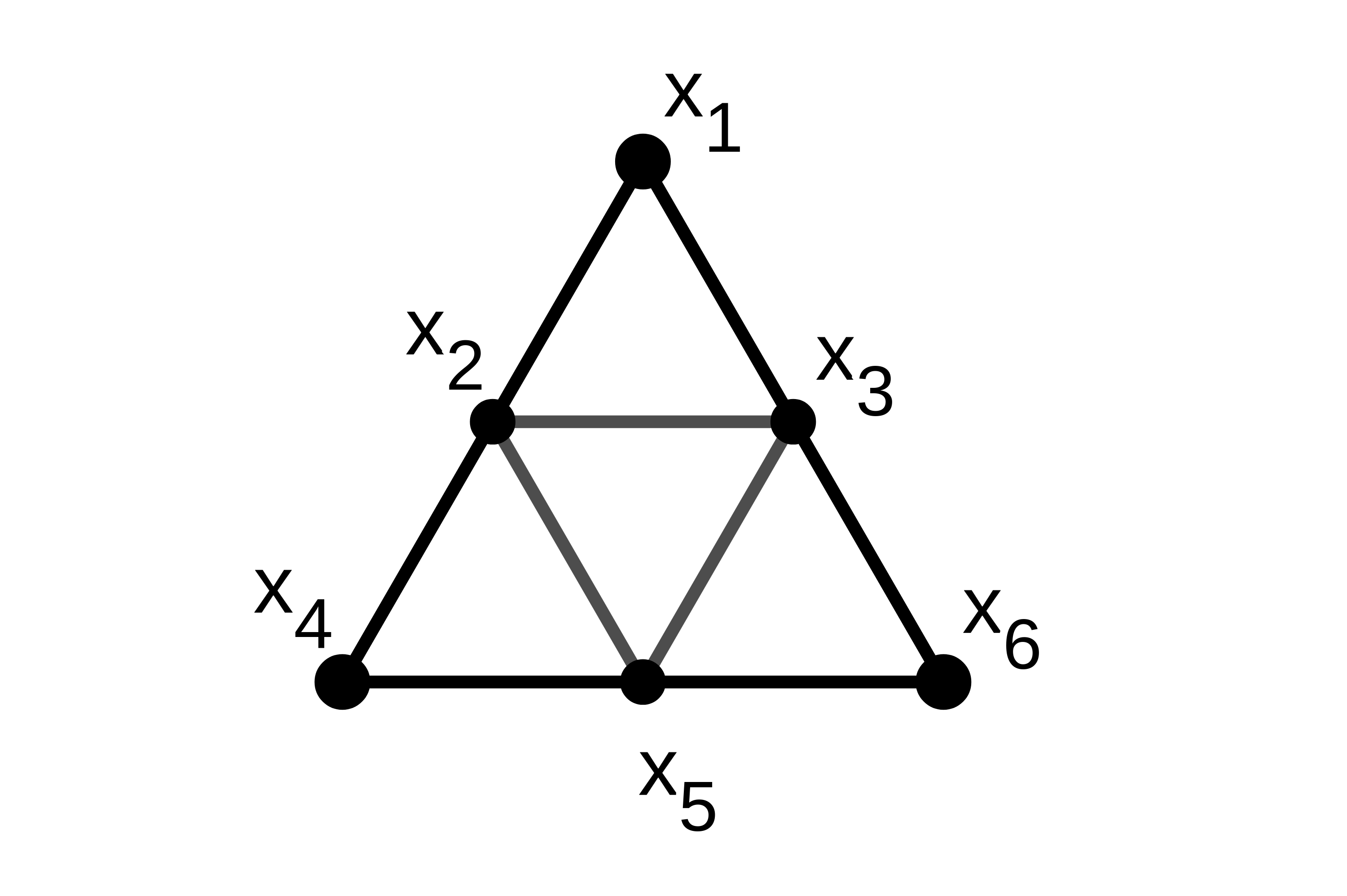}
  \caption{The graph $G$}\label{fig:graph}
 
 \end{minipage}
 \begin{minipage}{0.55\textwidth}

  \includegraphics[scale=0.4]{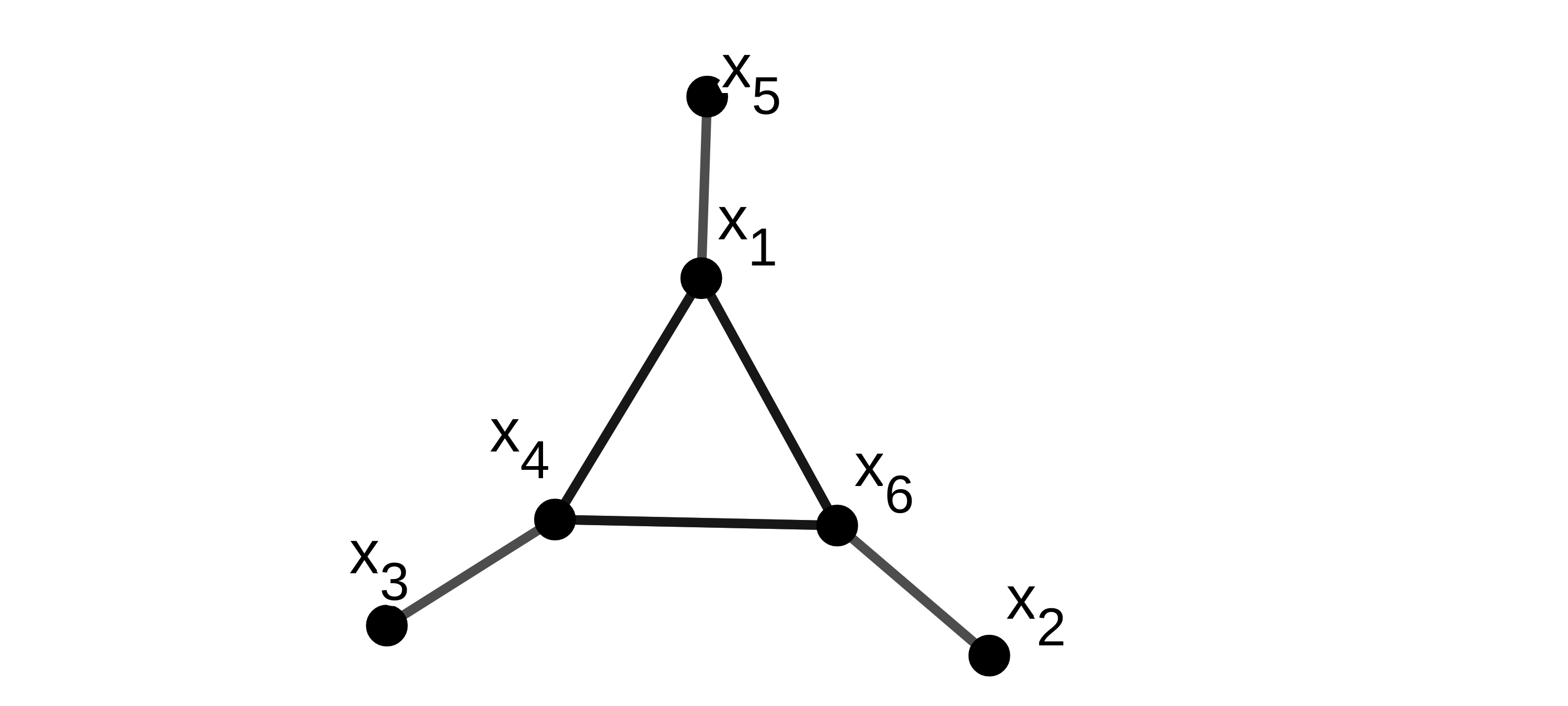}
  \caption{The complementary graph of $G$}\label{fig:graph_comp}         
               \end{minipage}
 \end{figure}
\end{exampl}

\subsection{The Betti numbers of the ideal generated by all square-free monomials of degree $n-1$}\label{subsec:square-free_n-1}
In this last section we will return to the square-free monomial ideals. Let $S=\mathbb{k}[x_1, \ldots, x_n]$, and let $I\subset S$ be the ideal generated by all square-free monomials of degree $n-1$. We will use the Rees algebra $R(I)$ to compute the Betti numbers of $S/I^k$. 

Let $f_i=x_1 \cdots x_n/x_i$, i. e. the product of all the variables except $x_i$. Then $I=(f_1, \ldots, f_n)$. As before, we let $T=S[y_1, \ldots, y_n]$, and $\phi:T \to S[t]$ be defined by $x_i \mapsto x_i$ and $y_i \mapsto f_it$. Since $I$ is a polymatroidal ideal with the SEP, the Rees ideal $J= \ker \phi$ is generated in the degrees $(0,2)$ and $(1,1)$, by Theorem \ref{thm:rees}. In fact, there are no non-trivial relations in degree $(0,2)$, and we can see that $J$ is generated by all elements on the form $x_iy_i-x_jy_j$. A minimal generating set for $J$ is then 
\[
 x_1y_1-x_2y_2, x_1y_1-x_3y_3, \ldots, x_1y_1-x_ny_n.
\]
These elements constitutes a regular sequence of length $n-1$ in $T$, and we can then use the Koszul complex to find the bigraded minimal resolution
 \begin{align}\label{eq:T-res}\begin{split}
 0  \to T(-n+1&,-n+1)^{\binom{n-1}{n-1}} \to \ldots  \\
 \ldots \to & T(-2,-2)^{\binom{n-1}{2}}  \to T(-1,-1)^{\binom{n-1}{1}}  \to T(0,0)  \to T/J \to 0
\end{split}\end{align}
of $T/J \cong R(I)$. Now notice that $\phi$ maps $T_{(*,k)}$ to $I^k$, and the kernel is $J_{(*,k)}$. Hence $ [T/J]_{(*,k)} \cong I^k$ as $S$-modules. Also,
\[
 T(-a,-b)_{(*,k)} = \bigoplus_{\sum u_i = k-b} \!\!\!\! S(-a)y_1^{u_1} \cdots y_n^{u_n} \cong {S(-a)^{\binom{n+k-b-1}{n-1}}}_{\ \textstyle{.}}
\]
Hence the degree $(*,k)$ part of (\ref{eq:T-res}) gives us the minimal resolution
\begin{equation}\label{eq:Ik-res}
 0 \to S(-(n-1))^{\binom{k}{n-1}} \to \dots \to S(-i)^{\binom{n+k-i-1}{n-1}\binom{n-1}{i}} \to \dots \to S(0)^{\binom{n+k-1}{n-1}} \to I^k \to 0
\end{equation}
of $I^k$.

\begin{theorem}
 Let $S=\mathbb{k}[x_1, \ldots, x_n]$, and let $I\subset S$ be the ideal generated by all square-free monomials of degree $n-1$. Then the Betti numbers of $S/I^k$ are given by
 \[
  \beta_i(S/I^k)=\binom{n+k-i}{n-1} \binom{n-1}{i-1}_{ \textstyle{.}}
 \]
\end{theorem}
\begin{proof}
 By (\ref{eq:Ik-res}) we have the minimal resolution
 \[
  0 \to S^{\binom{k}{n-1}} \to \dots \to S^{\binom{n+k-i-1}{n-1}\binom{n-1}{i}} \to \dots \to S^{\binom{n+k-1}{n-1}} \to S \to S/I^k \to 0
 \]
of $S/I^k$.
  \end{proof}

\section{Further questions and problems}
Lastly, we post some problems that arise from the topics of this paper. We will refer to the classes {\bf A} and {\bf B} introduced in Section 1. 

\begin{enumerate}
 \item In Section \ref{sec:betti} we computed the Betti numbers for the ideals in {\bf A} with $d=2,3$ and $n-1$, and the ideals in {\bf B} with $s=n-1$, and $n-2$. Can we cover more instances of these two classes, by combining Theorem \ref{thm:main} with some other method? Are there other classes for which we can use Theorem \ref{thm:main} to compute some of the Betti numbers?  
 \item The main question that we came across in Section \ref{sec:rees} is whether all polymatroidal ideals have their Rees ideals generated in the degrees $(0,2)$ and $(1,1)$. One may also ask if the Rees algebras of the ideals in Theorem \ref{thm:rees} are Koszul, as the defining ideals are generated in total degree two.
 \item Is there a nice description of all monomial ideals with linear powers, which have their Rees ideals generated in degrees $(0,2)$ and $(1,1)$? For example, one can try to classify all edge ideals with this property. 
\end{enumerate}

\section*{Acknowledgements}
The author would like to thank Ralf Fröberg, Christian Gottlieb, and Samuel Lundqvist for the rewarding discussions about the topics of this paper.

%

\end{document}